\documentclass[article,12pt]{amsart}
\usepackage[utf8]{inputenc}
\usepackage[T1]{fontenc}
\usepackage[english]{babel}
\usepackage[margin=2.8cm]{geometry}
\usepackage{mathpazo}
\usepackage[scaled=.95]{helvet}
\usepackage{courier}
\usepackage{amsmath,amsfonts,amssymb,amsthm}
\usepackage{graphicx}
\usepackage{float}
\usepackage{color}
\usepackage[onehalfspacing]{setspace}

\newcommand{\dF}{\mathbb{F}}
\newcommand{\dR}{\mathbb{R}}
\newcommand{\dE}{\mathbb{E}}

\newcommand{\cB}{\mathcal{B}}

\newcommand{\cN}{\mathcal{N}}
\newcommand{\cL}{\mathcal{L}}

\newcommand{\cF}{\mathcal{F}}

\newcommand{\rI}{\mathrm{I}}
\newcommand{\bkR}{\mathbb{R}}

\newcommand{\wh}{\widehat}

\renewcommand{\leq}{\leqslant}
\renewcommand{\geq}{\geqslant}
\def\videbox{\mathbin{\vbox{\hrule\hbox{\vrule height1ex \kern.5em
\vrule height1ex}\hrule}}}
\def\demend{\hfill $\videbox$\\}

\theoremstyle{plain}
\numberwithin{equation}{section}
\newtheorem{thm}{Theorem}[section]
\newtheorem{rem}{Remark}[section]

\newtheorem{lem}{Lemma}[section]

\newcommand{\limite}[2]{\xrightarrow[#1]{#2}}
\DeclareMathOperator{\cvl}{\limite{n \to \infty}{\cL}}

\DeclareUnicodeCharacter{B0}{\textdegree}

\usepackage[draft]{fixme}
\FXRegisterAuthor{h}{ah}{H}
\FXRegisterAuthor{p}{ap}{P}
\FXRegisterAuthor{a}{aa}{A}

\usepackage{hyperref}
\hypersetup{
pdfpagemode=UseOutlines,      
pdfstartview=Fit,             
pdffitwindow=true,            
pdfpagelayout=TwoColumnsRight,
pdftoolbar=true,              
pdfmenubar=true,              
bookmarksopen=false,          
bookmarksnumbered=true,       
colorlinks=true,              
pdfauthor={Ton nom},     
pdftitle={Titre PDF},    
pdfcreator=PDFLaTeX,          %
pdfproducer=PDFLaTeX,         %
linkcolor=blue,               
urlcolor=blue,                
anchorcolor=blue,            
citecolor=blue,              
frenchlinks=true,             
pdfborder={0 0 0}             
}
	\usepackage{xr}

\title[Stochastic Newton algorithm for logistic regressions]{An efficient stochastic Newton algorithm for parameter estimation in logistic regressions}
\author{Bernard Bercu}
\author{Antoine Godichon}
\author{Bruno Portier}

\begin{document}
\begin{abstract}
Logistic regression is a well-known statistical model which is commonly used in the situation where the output is a binary random variable. 
It has a wide range of applications including machine learning, public health, social sciences, ecology and econometry. In order to estimate 
the unknown parameters of logistic regression with data streams arriving sequentially and at high speed, we focus our attention on a recursive 
stochastic algorithm. More precisely, we investigate the asymptotic behavior of a new stochastic Newton algorithm. It enables to easily update 
the estimates when the data arrive sequentially and to have research steps in all directions. We establish the almost sure convergence of our
stochastic Newton algorithm as well as its asymptotic normality. All our theoretical results are illustrated by numerical experiments.
\end{abstract}
\maketitle

\section{Introduction}\label{secintro}
\def\bkE{\mathbb{E}}
\def\bkR{\mathbb{R}}
\def\wtheta{\widehat\theta}
\def\pa#1{\left(#1\right)}
\def\cro#1{\left[#1\right]}
\def\cB{\mathcal{B}}
\def\cL{\mathcal{L}}
\def\cR{\mathcal{R}}


Logistic regression is a well-known statistical model which is commonly used in the situation where the output is a binary random variable. 
It has a wide range of applications including machine learning \cite{bach2014adaptivity}, public health \cite{merlo2016original}, social sciences, 
ecology \cite{komori2016asymmetric} and econometry \cite{varian2014big}.
In what follows, we will consider a sequence $\left( X_{n} , Y_{n} \right)$ of random variables taking values in $\mathbb{R}^{d}\times \lbrace 0 ,1 \rbrace$, 
and we will assume that $\left( X_{n} \right)$ is a sequence of independent and identically distributed random vectors such that, that for all $n \geq 1$, 
the conditional distribution of $Y_{n}$ knowing $X_{n}$ is a Bernoulli distribution \cite{hosmer2013applied}.
More precisely, let $\theta = (\theta_0,  \ldots, \theta_d)^T$ be the unknown parameter belonging to $\bkR^{d+1}$ of the logistic regression.
For all $n \geq 1$, we denote $\Phi_n = (1, X_n^{T})^T$ and we assume that
$$
\cL\bigl(Y_{n} |  \Phi_n\bigr) =\cB \bigl( \pi ( \theta^T \Phi_n) \bigr)  \hspace{1cm} \text{where}  \hspace{1cm}  \pi(x)= \frac{\exp (x)}{1+\exp(x)}.
$$
Our goal is the estimation of the vector of parameters $\theta$. For that purpose, let $G$ be the convex positive function defined, for all $h \in \bkR^{d+1}$, by
\begin{eqnarray*}
G(h) &=& \bkE\Bigl[-\log\bigl(\pi(h^T \Phi)^Y  \bigl(1 - \pi(h^T \Phi)\bigr)^{1 - Y} \bigr) \Bigr], \\
&=& \bkE\Bigl[\log \bigl( 1+ \exp ( h^T \Phi ) \bigr) - h^T\Phi Y \Bigr]
\end{eqnarray*}
where $\cL(Y | \Phi)=\cB\bigl(\pi(\theta^T \Phi)\bigr)$ and $\Phi$ shares the same distribution as $\Phi_1$. We clearly have $\bkE\cro{Y | \Phi} = \pi(\theta^T \Phi)$. Hence, one can easily check
that the unknown parameter $\theta$ satisfies
\begin{equation}
\nabla G(\theta) = \bkE\Bigl[\Phi (\pi(\theta^T \Phi) - Y)\Bigr] = 0.
\end{equation}
Consequently, under some standard convexity assumptions on $G$,
\begin{equation}
\theta = \arg\min_{h \in \bkR^{d+1}} G(h).
\end{equation}
Since there is no explicit solution of the equation $\nabla G(h) = 0$, it is necessary to make use of an approximation algorithm in order to estimate $\theta$.
Usually, when the sample size is fixed, we approximate the solution with the help of a Newton root-finding numerical algorithm. 
However, when the data streams arrive sequentially and at high speed, it is much more appropriate and efficient to treat them 
with the help of stochastic gradient algorithms. We refer the reader to the seminal paper  \cite{robbins1951} and to its averaged version
\cite{PolyakJud92, ruppert1988efficient}, as well as to the more recent contributions on the logistic regression 
\cite{bach2014adaptivity,godichon2016,gadat2017optimal}. One can observe that in these last references, the conditional distribution 
$\cL(Y | \Phi)=\cR\bigl(\pi(\theta^T \Phi)\bigr)$ is the Rademacher distribution, instead of the usual Bernoulli $\cB\bigl(\pi(\theta^T \Phi)\bigr)$ one.
\ \vspace{2ex} \\
In this paper, we propose an alternative strategy to stochastic gradient algorithms, in the spirit of the Newton algorithm, in the sense that the step sequence 
of stochastic gradient algorithms is replaced by recursive estimates of the inverse of the Hessian matrix of the function we are minimizing. 
This strategy enables us to properly deal with the situation where the Hessian matrix has eigenvalues with significantly different absolute values. 
Indeed, in that case, it can be necessary to adapt automatically the step of the algorithm in all directions. 
To be more precise, we propose to estimate the unknown parameter $\theta$ with the help of a stochastic Newton algorithm given, for all $n \geq 1$, by
\begin{eqnarray*}
a_n & = &\pi(\theta_{n-1}^T \Phi_n) \pa{1 - \pi(\theta_{n-1}^T \Phi_n)}, \\
S_n^{-1} & = & S_{n-1}^{-1} - a_n (1 + a_n \Phi_n^T S_{n-1}^{-1}\Phi_n)^{-1}
S_{n-1}^{-1}\Phi_n\Phi_n^T S_{n-1}^{-1},\\ 
\theta_n & = & \theta_{n-1}\,+\,  S_{n}^{-1}\Phi_n \pa{Y_n - \pi(\theta_{n-1}^T \Phi_n)} 
\end{eqnarray*}
where the initial value $\theta_0$ is a bounded vector of $\dR^{d+1}$ which can be arbitrarily chosen and 
$S_0$ is a positive definite and deterministic matrix. For the sake of simplicity and in all the sequel, we take
$S_0 = I_{d+1}$ where $I_{d+1}$ stands for the identity matrix of order $(d+1)$. One can observe that
\begin{equation*}
S_{n} =   \sum_{k=1}^{n}a_{k} \Phi_{k}\Phi_{k}^{T}  + I_{d+1} 
\end{equation*}
Moreover, $S_n^{-1}$ is updated recursively, thanks to Riccati's equation (\cite{Duf97}, page 96)
which enables us to avoid the useless inversion of a matrix at each iteration of the algorithm. Furthermore,
the matrix  $n^{-1} S_n$ is an estimate of the Hessian matrix $ \nabla^{2} G \left( \theta \right)$ at the unknown value $\theta$.
In order to ensure the convergence of the stochastic Newton algorithm, a modified 
version of this algorithm is provided. We shall prove its asymptotic efficiency by establishing its almost sure convergence 
and its asymptotic normality.
\ \vspace{2ex} \\
This algorithm is closely related to the iterative one used to estimate the unknown vector $\theta$ of a linear regression model satisfying, for all $n \geq 1$,
$Y_n = \theta^T\Phi_n + \varepsilon_n$. As a matter of fact, the updating of the least squares estimator of the parameter $\theta$ is given by
\begin{eqnarray*}
S_n^{-1} &= & S_{n-1}^{-1} - (1 + \Phi_n^T  S_{n-1}^{-1} \Phi_n)^{-1}
S_{n-1}^{-1} \Phi_n\Phi_n^T S_{n-1}^{-1} \\
\theta_n &= &\theta_{n-1} + S_{n}^{-1}\Phi_n\pa{Y_n - \theta_{n-1}^T \Phi_n}  
\end{eqnarray*}
where the initial value $\theta_0$ can be arbitrarily chosen and $S_0$ is a positive definite deterministic matrix.
This algorithm can be considered as a Newton stochastic algorithm since the matrix $n^{-1} S_n$ is an estimate of the Hessian matrix of the least squares criterion
$\bkE\cro{(Y - \theta^T \Phi )^2}/2$.
\ \vspace{2ex} \\
To the best of our knowledge and apart from the least squares estimate mentioned above, stochastic Newton algorithms are hardly ever used and studied since they often require 
the inversion of a matrix at each step, which can be very expensive in term of time calculation. An alternative to the stochastic Newton algorithm is the BFGS algorithm
\cite{mokhtari2014res,lucchi2015variance,byrd2016stochastic}
based on the recursive estimation of a matrix whose behavior is closed to the one of the inverse of the Hessian matrix. 
Nevertheless, this last estimate does not converge to the exact inverse of the Hessian matrix. Consequently, the estimation of the unknown vector $\theta$ 
is not satisfactory.
\ \vspace{2ex} \\
The paper is organized as follows. Section \ref{secframe} describes the framework and our main assumptions.
In Section \ref{secalgo}, we introduce our new stochastic Newton algorithm.
Section \ref{vit} is devoted to its almost sure convergence as well as its asymptotic normality. Our theoretical results are
illustrated by numerical experiments in Section \ref{sim}. Finally, all technical proofs are postponed to Sections \ref{secproofsascvg}
and \ref{secproofsan}.

\section{Framework}\label{secframe}

In what follows, we shall consider a couple of random variables $( X,Y)$ taking values in $\bkR^{d} \times \lbrace 0,1 \rbrace$ where $d$ is a positive integer, and such that 
$$
\cL\bigl(Y |  \Phi\bigr) =\cB \bigl( \pi ( \theta^T \Phi) \bigr)  \hspace{1cm} \text{where}  \hspace{1cm}  \pi(x)= \frac{\exp (x)}{1+\exp(x)}
$$
with $\Phi = ( 1, X^{T})^{T}$ and where $\theta = \left( \theta_{0},\theta_{1},...,\theta_{d} \right)$ is the unknown parameter to estimate.  
We recall that $\theta$ is a minimizer of the convex function $G$ defined, for all $h \in \bkR^{d+1}$, by
\begin{equation}
\label{DEFG}
G(h) = \bkE\Bigl[-\log\bigl(\pi(h^T \Phi)^Y  \bigl(1 - \pi(h^T \Phi)\bigr)^{1 - Y} \bigr) \Bigr]=\bkE\bigl[ g \left( \Phi , Y, h \right) \bigr] .
\end{equation}
In all the sequel, we assume that the following assumptions are satisfied.
\begin{itemize}
\item[\textbf{(A1)}] The vector $\Phi$ has a finite 
moment of order $2$ and
the matrix $\bkE[ \Phi \Phi^{T}]$ is positive definite.
\item[\textbf{(A2)}] The Hessian matrix $\nabla^{2}G \left( \theta \right)$ is positive definite.
\end{itemize}
These assumptions ensure that $\theta$ is the unique minimizer of the functional $G$. Assumption \textbf{(A1)} enables us to find a first lower bound for the smallest eigenvalue of the estimates of 
the Hessian matrix, while assumptions \textbf{(A1)} and \textbf{(A2)} give the unicity of the minimizer of $G$ and ensure that the functional $G$ is twice continuously differentiable.
More precisely, for all $h \in \bkR^{d+1}$, we have
\begin{eqnarray}
\label{GG}
 \nabla G (h) & = & \bkE\left[ \nabla_{h} \,g \left( \Phi,Y,h \right)\right] = \bkE \left[ \frac{\exp \left( h^{T}\Phi\right)}{1+ \exp \left( h^{T}\Phi \right)}\Phi \right] - \bkE\left[ Y\Phi \right] , \\
 \nabla^{2}G \left( h \right) & = & \bkE\left[ \nabla_{h}^{2}\,g \left( \Phi , Y, h \right)\right]   =  \frac{1}{4}\,\bkE\left[ \frac{1}{\left( \cosh \left( h^{T}\Phi /2 \right) \right)^{2}} \Phi \Phi^{T} \right] .
\label{HG}
\end{eqnarray}

\begin{rem}
In the previous literature, it is more usual to consider a variable $Y$ taking values in $\lbrace -1,1\rbrace$,  which means that
$\cL(Y | \Phi)=\cR\bigl(\pi(\theta^T \Phi)\bigr)$ is the Rademacher distribution
\cite{bach2014adaptivity,godichon2016,gadat2017optimal}. In this context, $\theta$  is a minimizer of the functional $G$ defined, for all $h \in \bkR^{d+1}$, by
\begin{equation*}
G (h) = \bkE\Bigl[\ln \left( 1+ \exp \left( - Y h^{T} \Phi \right) \right)\Bigr].
\end{equation*}
Under assumptions, the functional $G$ is twice continuously differentiable and, for all $h \in \bkR^{d+1}$,
\begin{eqnarray*}
 \nabla G(h)   & = & -\bkE\left[ \frac{\exp \left( -Yh^{T}\Phi \right)}{1+ \exp \left( -Y h^{T}\Phi \right)}Y\Phi \right], \\
  \nabla^{2}G \left( h \right) & = & \frac{1}{4}\,\bkE\left[ \frac{1}{\left( \cosh \left( h^{T}\Phi /2 \right) \right)^{2}} \Phi \Phi^{T} \right] .
\end{eqnarray*}
One can observe that the Hessian matrix  remains the same. It ensures that the algorithm introduced in Section \ref{secalgo} can be adapted to this functional and 
that all the results given in Section \ref{vit} still hold in this case.
\end{rem}

\section{Stochastic Newton algorithm}\label{secalgo}
In order to deal with massive data acquired online, let us recall that the stochastic Newton algorithm presented in the introduction is
given, for all $n \geq 1$, by
\begin{eqnarray}
\label{Defaini}
a_n & = &\pi(\theta_{n-1}^T \Phi_n) \pa{1 - \pi(\theta_{n-1}^T \Phi_n)}, \\
\label{DefISini}
S_n^{-1} & = & S_{n-1}^{-1} - a_n (1 + a_n \Phi_n^T S_{n-1}^{-1}\Phi_n)^{-1}
S_{n-1}^{-1}\Phi_n\Phi_n^T S_{n-1}^{-1},\\ 
\label{Defthetaini}
\theta_n & = & \theta_{n-1}\,+\,  S_{n}^{-1}\Phi_n \pa{Y_n - \pi(\theta_{n-1}^T \Phi_n)} 
\end{eqnarray}
where the initial value $\theta_0$ is a bounded vector of $\dR^{d+1}$ which can be arbitrarily chosen and $S_0 = I_{d+1}$.
Unfortunately, we were not able to prove that $n^{-1} S_n$ converges almost surely to the Hessian matrix $ \nabla^{2} G \left( \theta \right)$,
as well as to establish the almost sure convergence of $\theta_{n}$
to $\theta$. This is the reason why we slightly modify our strategy by proposing
a truncated version of previous estimates given, for all $n \geq 1$, by
\begin{eqnarray}
\label{Defan}
\wh{a}_n & = &\pi(\wh{\theta}_{n-1}^T \Phi_n) \pa{1 - \pi(\wh{\theta}_{n-1}^T \Phi_n)} \\
\label{Defthetan}
\wh{\theta}_n & = & \wh{\theta}_{n-1}\,+\,  S_{n-1}^{-1}\Phi_n \pa{Y_n - \pi(\wh{\theta}_{n-1}^T \Phi_n)} \\
\label{DefISn}
S_n^{-1} & = & S_{n-1}^{-1} - \alpha_n (1 + \alpha_n \Phi_n^T S_{n-1}^{-1}\Phi_n)^{-1}
S_{n-1}^{-1}\Phi_n\Phi_n^T S_{n-1}^{-1}
\end{eqnarray}
where the initial value $\wh{\theta}_0$ is a bounded vector of $\dR^{d+1}$ which can be arbitrarily chosen, $S_0 = I_{d+1}$
and $(\alpha_n)$ is a sequence of random variable defined, for some positive constant $c_\alpha$, by
\begin{equation}
\label{defalphan}
\alpha_{n} = \max \Bigl\lbrace \wh{a}_{n}, \frac{c_{\alpha}}{n^{\beta}} \Bigr\rbrace = 
 \max \left\lbrace \frac{1}{4\bigl( \cosh \bigl( \Phi_{n}^{T}\wh{\theta}_{n-1}/2 \bigr)\bigr)^{2}}, \frac{c_{\alpha}}{n^{\beta}} \right\rbrace 
\end{equation}
with $\beta \in ] 0 ,1/2 [$. From now on and for the sake of simplicity, we assume that $c_{\alpha} \leq 1/4$.
It immediately implies that, for all $n \geq 1$, $\alpha_n \leq 1/4$. However, the proofs remains true for any $c_{\alpha}>0$. 
We already saw in Section \ref{secintro} that $S_{n}^{-1}$ coincides with the exact inverse of the weighted matrix $S_{n}$ given, for all $n \geq 1$, by 
\begin{equation}
\label{defSn}
\vspace{-1ex}
S_{n} =   \sum_{k=1}^{n}\alpha_{k} \Phi_{k}\Phi_{k}^{T} + I_{d+1}.
\end{equation}
Moreover, we will see in Section \ref{vit} that, even with this truncation of the estimate of 
the Hessian matrix, $n^{-1} S_n$ converges almost surely to the Hessian matrix $ \nabla^{2} G \left( \theta \right)$.
Consequently, we will still have an optimal asymptotic behavior of the estimator $\wh{\theta}_{n}$
of $\theta$.

\section{Main results}\label{vit}

Our first result deals with the almost sure convergence of our estimates of $\theta$ and the Hessian matrix $\nabla^{2}G \left( \theta \right)$. For all $n\geq 1$,
denote
$$
\overline{S}_{n} = \frac{1}{n}S_n.
$$
\begin{thm}
\label{T-ASCVG}
Assume that \textbf{(A1)} and \textbf{(A2)} are satisfied. Then, we have the almost sure convergences
\vspace{-1ex}
\begin{equation}
\label{T-ascvg1}
 \lim_{n \to \infty}\, \wh{\theta}_{n} = \theta  \hspace{1cm}\text{a.s}
\end{equation}
\begin{equation}
\label{T-ascvg2}
\lim_{n \to \infty} \, \overline{S}_{n}= \nabla^{2}G \left( \theta \right) \hspace{1cm}\text{a.s}
\end{equation}
\end{thm}

\noindent
We now focus on the almost sure rates of convergence of our estimate of $\theta$.

\begin{thm}
\label{T-RATE}
Assume that \textbf{(A1)} and \textbf{(A2)} are satisfied. Then, we have for all $\gamma > 0$,
\begin{equation}
\label{T-rate1}
\bigl\| \wh{\theta}_{n}-\theta\bigr\|^{2} = o \left( \frac{(\log n )^{1+\gamma} }{n} \right) \hspace{1cm}\text{a.s}
\end{equation}
Moreover, suppose the random vector $\Phi$ has a finite moment of order $>2$. Then, we have
\begin{equation}
\label{T-rate2}
\bigl\| \wh{\theta}_{n}-\theta\bigr\|^{2} = O \left( \frac{\log n }{n} \right)  \hspace{1cm}\text{a.s}
\end{equation}
\end{thm}

\noindent
The almost sure rates of convergence of our estimate of the Hessian matrix $\nabla^{2}G \left( \theta \right)$ and its inverse are as follows.

\begin{thm}
\label{T-RATEH}
Assume that \textbf{(A1)} and \textbf{(A2)} are satisfied and that the random vector $\Phi$ has a finite moment of order $4$. Then, we have 
for all $0<\beta<1/2$,
\begin{equation}
\label{T-rate3}
 \bigl\| \overline{S}_{n} - \nabla^{2}G \left( \theta \right) \bigr\|^{2} = O \left( \frac{1}{n^{2\beta}} \right) \hspace{1cm}\text{a.s}
\end{equation}
In addition, we also have 
\begin{equation}
\label{T-rate4}
 \Bigl\| \overline{S}_{n}^{-1} - \left( \nabla^{2}G \left( \theta \right) \right)^{-1} \Bigr\|^{2} = O \left( \frac{1}{n^{2\beta}} \right) \hspace{1cm}\text{a.s}
\end{equation}
\end{thm}
\vspace{-2ex}
\begin{proof}
The proofs of Theorems \ref{T-ASCVG}, \ref{T-RATE} and \ref{T-RATEH} are given in Section \ref{secproofsascvg}.
\end{proof}

\begin{rem}
One can observe that we do not obtain the parametric rate $1/n$ for these estimates. This is due to the truncation $\alpha_n$ which
slightly modifies our estimation procedure. However, without this truncation, we were not able to establish the almost
sure convergence of any estimate. Finally, the last result \eqref{T-rate4} ensures that our estimation procedure performs
pretty well and that the estimator $\wh{\theta}_n$ has an optimal asymptotic behavior.
\end{rem}

\begin{thm}
\label{T-AN}
Assume that \textbf{(A1)} and \textbf{(A2)} are satisfied and that the random vector $\Phi$ has a finite moment of order $4$. 
Then, we have the asymptotic normality
\begin{equation}
\label{T-an1}
 \sqrt{n} \left( \wh{\theta}_{n} - \theta \right) \cvl \mathcal{N}\Bigl( 0 , \left(\nabla^{2} G \left( \theta \right) \right)^{-1} \Bigr) .
\end{equation}
\end{thm}

\begin{proof}
The proof of Theorem \ref{T-AN} is given in Section \ref{secproofsan}.
\end{proof}

\begin{rem}
We deduce from \eqref{T-ascvg2} and \eqref{T-an1} that
\begin{equation}
\label{T-an2}
\pa{\wh{\theta}_{n}  - \theta}^T S_n \pa{\wh{\theta}_{n}  - \theta}
\cvl  \chi^2(d+1).
\end{equation}
Convergence \eqref{T-an2} allows us to build confidence regions for the parameter $\theta$.
Moreover, for any vector $w \in\mathbb{R}^{d+1}$ different from zero, we also have
\begin{equation}
\label{T-an3}
\dfrac{w^T \pa{\wh{\theta}_n - \theta}}{\sqrt{w^T S_n^{-1} w}}
\cvl \cN(0,1).
\end{equation}
Confidence intervals and significance tests for the components of $\theta$  can be designed from \eqref{T-an3}. One can observe that
our stochastic Newton algorithm has the same asymptotic behavior as the averaged version of a stochastic gradient algorithm 
\cite{gadat2017optimal, GB2017, Pel00}.
\end{rem}

\section{Numerical experiments}\label{sim}

The goal of this section is to illustrate the asymptotic behavior of the truncated stochastic Newton algorithm (TSN) 
defined by equation \eqref{Defthetan}. 
For that purpose, we will focus on the model introduced in \cite{clemencon2015survey}
and used for comparing several gradient algorithms. We shall compare the numerical performances
of the TSN algorithm with those obtained with three different algorithms :  the Stochastic Newton (SN) algorithm
given by equation \eqref{Defthetaini}, the stochastic gradient algorithm (SG), and the averaged stochastic gradient algorithm (ASG).  
Let us mention that simulations were carried out using the statistical software R.

\subsection{Experiment model}

We focus on the model introduced in \cite{clemencon2015survey}, defined  by
$$
\cL(Y|\Phi) =\cB \bigl(\pi ( \theta^T \Phi )\bigr) 
$$
where $\Phi= (1, X^{T})^T$ and $X$ is a random vector of $\dR^d$ with $d=10$
with independent coordinates uniformly distributed on the interval $[0 ,1]$. Moreover the unknown parameter
$\theta = (-9,0,3,-9,4,-9,15,0,-7,1,0)^T$. This model is particularly interresting since it leads
to a Hessian matrix $\nabla^2 G(\theta)$ with eigenvalues of different order sizes.
Indeed, one can see in Table \ref{Tab_EigenH} that the smallest eigenvalue of $\nabla^2 G(\theta)$ is close to 4.422 $10^{-4}$ while
its largest eigenvalue is close to 0.1239.

\vspace{-2ex}
\begin{table}[H]\centering
\begin{tabular}{|c|c|c|c|c|c|}
\hline
0.1239 & 2.832 $10^{-3}$ & 2.822 $10^{-3}$ & 2.816 $10^{-3}$ & 
 2.778 $10^{-3}$ & 2.806 $10^{-3}$  \\
\hline
 2.651 $10^{-3}$ & 2.517 $10^{-3}$ & 2.1567 $10^{-3}$ & 9.012 $10^{-4}$ & 4.422 $10^{-4}$ &
\\
 \hline
\end{tabular}
\caption{Estimated eigenvalues of $\nabla^2G(\theta )$
arranged in decreasing order.}
\label{Tab_EigenH}
\end{table}
\vspace{-2ex}

\subsection{Comparison of the different algorithms}

Our comparisons are based on the mean squared error (MSE) defined, for all estimate $\widehat{\theta}_{n}$ of $\theta$, by
$$
\dE\Bigl[ \bigl\| \widehat{\theta}_{n} - \theta \bigr\|^{2} \Bigr].
$$
We simulate $N=400$ samples wth a maximum number of iterations $n=5\,000$. 
For each sample, we estimate the unknown parameter $\theta$ using the four algorithms (TSN, SN, SG, ASG)
which are initialized identically by choosing the initial value $\widehat{\theta}_0$ uniformly in a compact subset containing the true value $\theta$. 
For the TSN and SN algorithms, we take $S_0=I_{d+1}$. In addition, 
for the TSN algorithm,  we choose the truncation term defined by
$c_{\alpha}=10^{-10}$  and $\beta =0.49$.
Finally, to be fairplay, we choose the best step sequence for the SG algorithm with the help of a cross validation method. 
Figure \ref{FigMSE} shows the decreasing behavior of the MSE, calculated for the four algorithms, as the number of iterations $n$ grows from $1$ to $5\,000$.

\begin{figure}[H]
\centering
\includegraphics[scale=0.6]{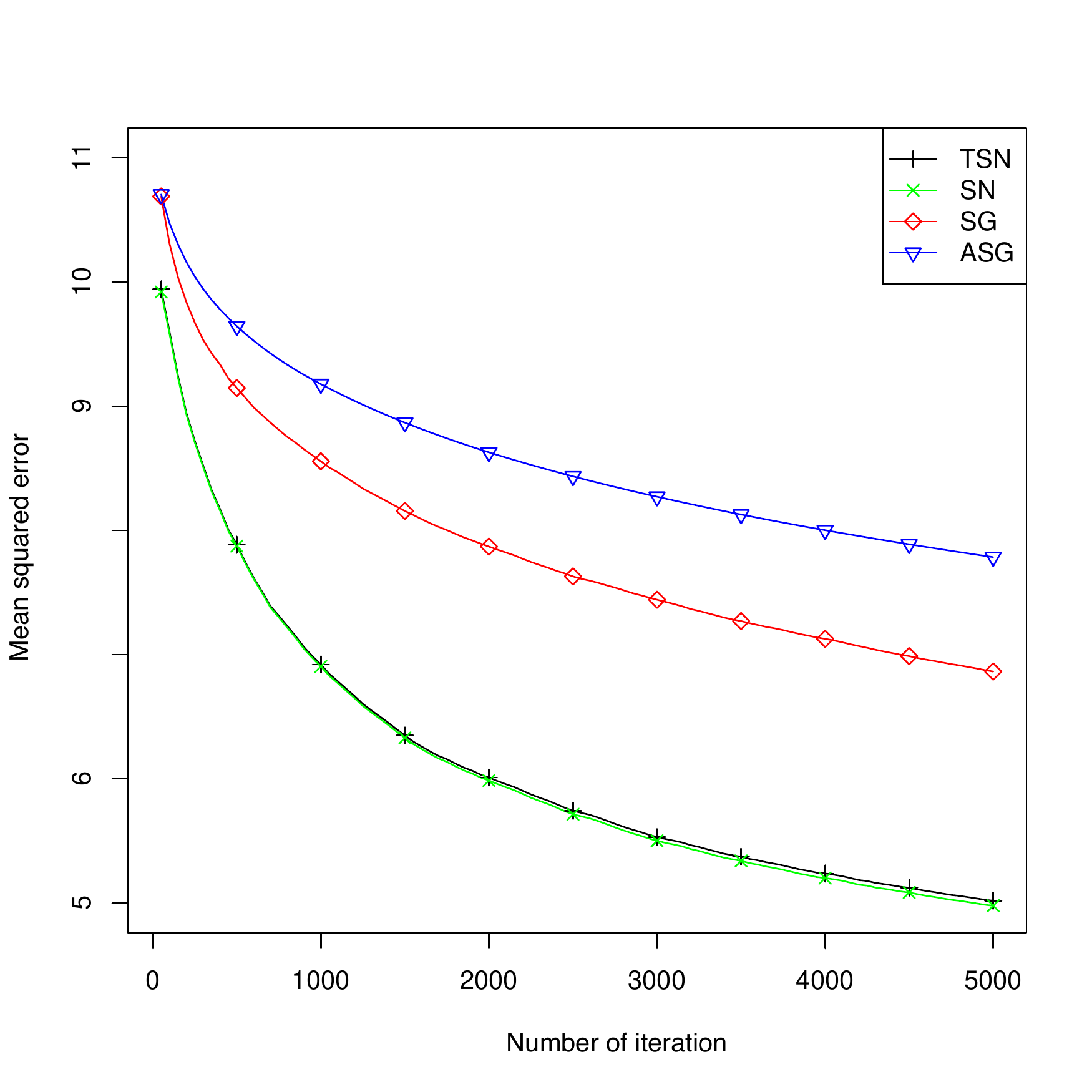}
\caption{Mean squared error of the four algorithms. }
\label{FigMSE}
\end{figure}

\noindent
It is clear that the stochastic Newton algorithms perform much more better than the stochastic gradient algorithms
The bad behavior of the stochastic gradient algorithms is certainly due to the fact that the eigenvalues of the Hessian
matrix $\nabla^2 G(\theta)$ are at different scales. One can also observe that it is quite useless to average the SG algorithm.
\ \vspace{1ex}\\
On can find in Figure \ref{FigBOX} the boxplots of the $N=400$ values of the squared error $\| \widehat{\theta}_{n} - \theta \|^{2}$ 
computed for the TSN and SN algorithms, as well as for the deterministic Newton-Raphson algorithm (NR).

\vspace{-5ex}
\begin{figure}[H]
\centering
\includegraphics[scale=0.6]{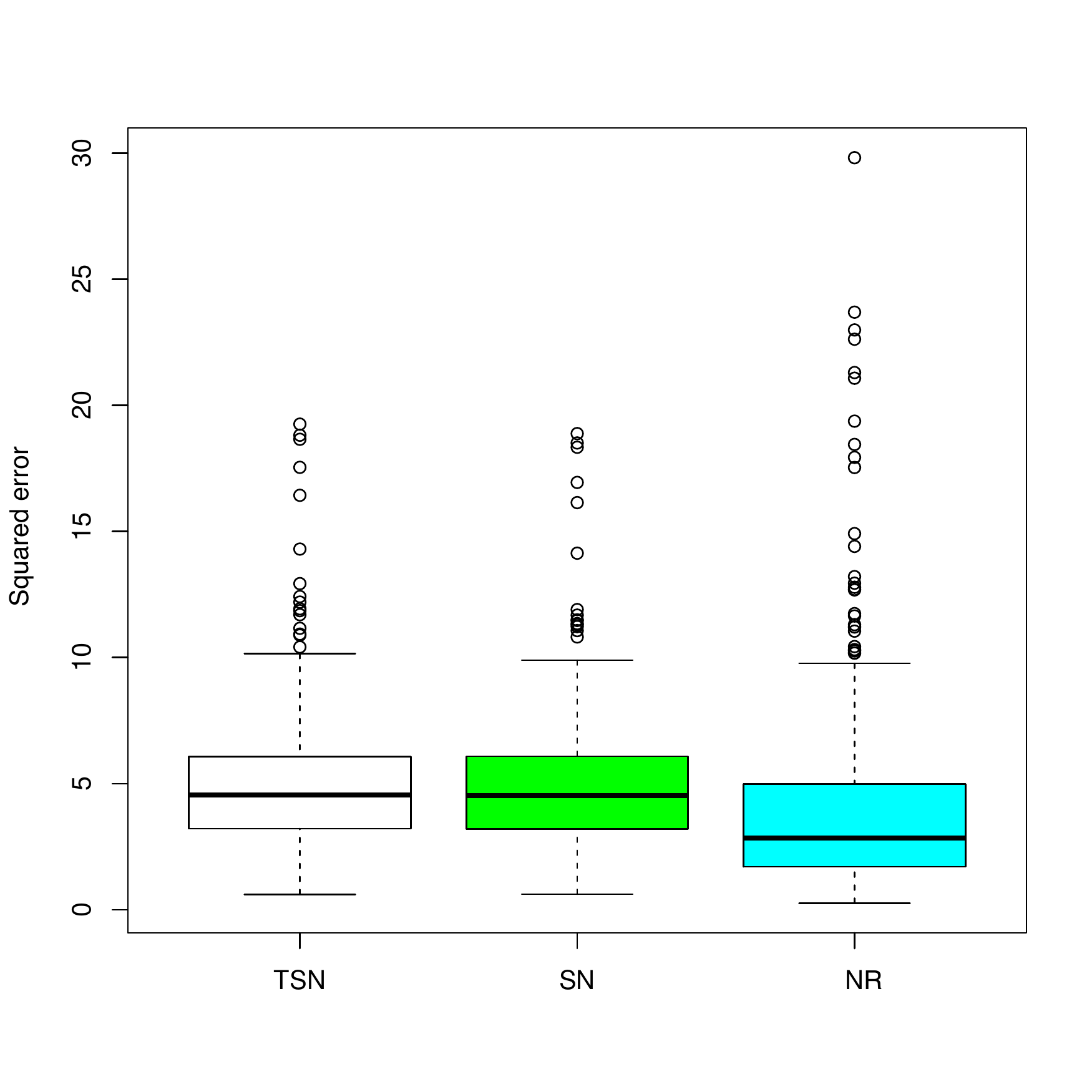}
\caption{Boxplots of the squared error for the TSN, SN and the NR algorithms.}\label{FigBOX}
\end{figure}

\subsection{Some comments concerning the truncation.}
To close this section, let us  make some comments concerning the truncation term introduced
in the TSN algorithm. This short numerical experiment tends to show that the use of the truncation is artificial and useless. 
Indeed, one can take the constant $c_\alpha$ in \eqref{defalphan} as small as possible and see that the TSN and SN algorithms match. 
Finally, an inappropriate choice of $c_\alpha$ can lead to a poor numerical behavior of the TSN algorithm. 

\newpage

\section{Proofs of the almost sure convergence results}
\label{secproofsascvg}

\subsection{Two technical lemmas}

We start the proofs of the almost sure convergence results with two technical lemmas.

\begin{lem}\label{LemAbel}
Assume that the random vector $\Phi$ has a finite moment of order $2$. Then, we have the almost sure convergence
for all $0<\beta<1$, 
\begin{equation}
\label{CvgAbel}
\lim_{n \to \infty}  \frac{1}{\sum_{k=1}^{n}k^{-\beta}}\sum_{k=1}^{n}k^{-\beta}\Phi_{k}\Phi_{k}^{T} = \dE\bigl[ \Phi \Phi^{T} \bigr]  \hspace{1cm}
\text{a.s}
\end{equation}
\end{lem}

\begin{rem}
We obtain from \eqref{defalphan} together with \eqref{defSn} that for all $n \geq 1$,
\begin{equation}
\label{LBSn}
S_n \geq c_\alpha\sum_{k=1}^{n}k^{-\beta}\Phi_{k}\Phi_{k}^{T}.
\end{equation}
Denote by $\lambda_{min}(S_n)$ the minimum eigenvalue of the positive definite matrix $S_n$. We immediately obtain from \eqref{CvgAbel} and \eqref{LBSn}
that for $n$ large enough
$$
\lambda_{min}(S_n) \geq \frac{\lambda c_\alpha}{2} \sum_{k=1}^{n}k^{-\beta} \geq \frac{\lambda c_\alpha}{2} n^{1-\beta}  \hspace{1cm}
\text{a.s}
$$
where $\lambda$ stands for the minimum eigenvalue of the positive definite deterministic matrix
$\dE\bigl[ \Phi \Phi^{T} \bigr] $. Consequently, we have under assumption \textbf{(A1)} that for $n$ large enough,
$$
\left( \lambda_{min}(S_n) \right)^{-2} \leq \frac{4}{\lambda^2 c_\alpha^2}  \frac{1}{n^{2(1- \beta)}}  \hspace{1cm}
\text{a.s}
$$
Therefore, as soon as $\beta \in ]0 ,1/2[$, 
\begin{equation}
\label{SumSn}
\sum_{n=1}^\infty \left( \lambda_{\min}\left( S_{n} \right) \right)^{-2} < \infty \hspace{1cm}
\text{a.s}
\end{equation}
\end{rem}

\begin{proof}
It follows from a straightforward Abel transform calculation that
\begin{eqnarray}
\sum_{k=1}^{n}k^{-\beta}\Phi_{k}\Phi_{k}^{T} &=& \sum_{k=1}^{n}k^{-\beta}\bigl(\Sigma_k - \Sigma_{k-1}\bigr), \nonumber \\
&=& n^{-\beta} \Sigma_n + \sum_{k=1}^{n-1}\bigl(k^{-\beta}-(k+1)^{-\beta} \bigr)\Sigma_k , \nonumber \\
&=& n^{-\beta} \Sigma_n +  \sum_{k=1}^{n-1}b_k k^{-1}\Sigma_k
\label{AbelSn}
\end{eqnarray}
where $\Sigma_0=0$ and for all $n\geq 1$, 
$$
\Sigma_n = \sum_{k=1}^n \Phi_{k}\Phi_{k}^{T}
\hspace{1cm}\text{and} \hspace{1cm}
b_n= n \bigl(n^{-\beta}-(n+1)^{-\beta}\bigr).
$$
On the one hand, we obtain from the standard strong law of large numbers that
\begin{equation}
\label{CvgSigman1}
\lim_{n \to \infty}  \frac{1}{n} \Sigma_n = \dE\bigl[ \Phi \Phi^{T} \bigr]  \hspace{1cm}
\text{a.s}
\end{equation}
On the other hand,
$$
\sum_{k=1}^n b_k= \sum_{k=1}^n k^{-\beta} - n^{1-\beta}
$$
which implies that
\begin{equation*}
\lim_{n \to \infty}  \frac{1}{n^{1-\beta}}  \sum_{k=1}^n b_k= \frac{\beta}{1-\beta}.
\end{equation*}
Then, we deduce from Toeplitz's lemma given e.g. in (\cite{Duf97}, page 54) that
\begin{equation}
\label{CvgSigman2}
\lim_{n \to \infty}  \frac{1}{n^{1-\beta}} \sum_{k=1}^{n-1}b_k k^{-1}\Sigma_k = \frac{\beta}{1-\beta}\dE\bigl[ \Phi \Phi^{T} \bigr]  \hspace{1cm}
\text{a.s}
\end{equation}
Consequently, we obtain from \eqref{AbelSn} together with \eqref{CvgSigman1} and \eqref{CvgSigman2} that
\begin{equation*}
\lim_{n \to \infty}  \frac{1}{n^{1-\beta}} 
\sum_{k=1}^{n}k^{-\beta}\Phi_{k}\Phi_{k}^{T}=\frac{1}{1-\beta}\dE\bigl[ \Phi \Phi^{T} \bigr]  \hspace{1cm}
\text{a.s}
\end{equation*}
which immediately leads to \eqref{CvgAbel}.
\end{proof}

\noindent
Our second lemma concerns a useful Lipschitz property of the function $\alpha$ defined, for all $h,\ell \in \mathbb{R}^{d+1}$, by
\begin{equation}
\label{defalpha}
\alpha \left( h,\ell \right) =  \pi \bigl( h^{T}\ell \bigr) \bigl( 1- \pi \bigl( h^{T}\ell \bigr) \bigr)
= \frac{1}{4 \left( \cosh \left( h^T \ell /2 \right) \right)^{2}}.
\end{equation}

\begin{lem}
\label{Lemalpha}
For all $h,\ell,\ell^\prime \in \mathbb{R}^{d+1}$, we have
\begin{equation}
\label{Lipalpha}
\left| \alpha \left( h,\ell \right) - \alpha \left( h,\ell^\prime \right)  \right| \leq \frac{1}{12\sqrt{3}} \left\| h \right\| \left\| \ell - \ell^\prime \right\| .
\end{equation}
\end{lem}

\begin{proof}
Let $\varphi$ be the function defined, for all $x \in \dR$, by
$$
\varphi(x)= \pi(x)(1-\pi(x))=\frac{\exp(x)}{(1+\exp(x))^2}=\frac{1}{4 ( \cosh ( x /2 ))^{2}}.
$$
We clearly have
\begin{eqnarray*}
\varphi^\prime(x)&=&\frac{\exp(x)(1-\exp(x))}{(1+\exp(x))^3}, \\
\varphi^{\prime \prime}(x)&=&\frac{\exp(x)((\exp(x))^2-4\exp(x)+1)}{(1+\exp(x))^4}.
\end{eqnarray*}
It is not hard to see that for all $x \in \dR$,
\begin{equation}
\label{Majvarphi}
| \varphi^\prime(x) | \leq \frac{1}{6\sqrt{3}}.
\end{equation}
Hence, it follows from \eqref{Majvarphi} together with the mean value theorem that for all $x,y \in \dR$,
\begin{equation}
\label{Lip1}
| \varphi(x) - \varphi(y) | \leq \frac{1}{6\sqrt{3}} | x-y |.
\end{equation}
Consequently, we obtain from \eqref{Lip1} that for all $h,\ell,\ell^\prime \in \mathbb{R}^{d+1}$, 
$$
\left| \alpha \left( h,\ell \right) - \alpha \left( h,\ell^\prime \right)  \right| \leq  \frac{1}{12\sqrt{3}} \bigl| h^T(\ell - \ell^\prime) \bigr| \leq 
\frac{1}{12\sqrt{3}} \left\| h \right\| \left\| \ell - \ell^\prime \right\|
$$
which completes the proof of Lemma \ref{Lemalpha}.
\end{proof}

\vspace{-2ex}
\subsection{Proof of Theorem \ref{T-ASCVG}.}

We are now in the position to proceed to the proof of the almost sure convergence 
\eqref{T-ascvg1}. By a Taylor expansion of the twice continuously differentiable functional $G$, there exists $\xi_n \in \mathbb{R}^{d+1}$ such that
\begin{equation}
\label{TaylorG1}
G \bigl( \wh{\theta}_{n+1}\bigr) = G \bigl( \wh{\theta}_{n} \bigr)  + \nabla G \bigl(  \wh{\theta}_{n} \bigr)^{T} 
\bigl(  \wh{\theta}_{n+1}- \wh{\theta}_{n} \bigr)  + \frac{1}{2}\bigl( \wh{\theta}_{n+1} -  \wh{\theta}_{n} \bigr)^{T} \nabla^{2}G ( \xi_n ) 
\bigl( \wh{\theta}_{n+1} -  \wh{\theta}_{n} \bigr) .
\end{equation}
We clearly have from \eqref{HG} that
$$
\left\| \nabla^{2}G (\xi_n) \right\| \leq \frac{1}{4} \dE\bigl[ \left\| \Phi \right\|^{2} \bigr].
$$ 
Hence, we obtain from \eqref{Defthetan} together with \eqref{TaylorG1} that
\begin{eqnarray*}
G\bigl( \wh{\theta}_{n+1}\bigr) & \leq &G \bigl( \wh{\theta}_{n} \bigr) + \nabla G \bigl(  \wh{\theta}_{n} \bigr)^{T} 
\bigl(  \wh{\theta}_{n+1}- \wh{\theta}_{n} \bigr) +\frac{1}{8} \dE\bigl[ \left\| \Phi \right\|^{2} \bigr]
\bigl\| \wh{\theta}_{n+1} -  \wh{\theta}_{n}  \bigr\|^{2}, \\
& = & G \bigl( \wh{\theta}_{n} \bigr) - \nabla G \bigl(  \wh{\theta}_{n} \bigr)^{T} S_n^{-1}Z_{n+1} +\frac{1}{8} \dE\bigl[ \left\| \Phi \right\|^{2} \bigr]
\bigl\| S_n^{-1}Z_{n+1} \bigr\|^{2}, \\
& \leq & G \bigl( \wh{\theta}_{n} \bigr) - \nabla G \bigl(  \wh{\theta}_{n} \bigr)^{T} S_n^{-1}Z_{n+1} +\frac{1}{8} \dE\bigl[ \left\| \Phi \right\|^{2} \bigr]
\left( \lambda_{\min} \left( S_{n} \right) \right)^{-2} \bigl\| Z_{n+1} \bigr\|^{2}
\end{eqnarray*}
where $Z_{n+1}=\nabla_h g \bigl( \Phi_{n+1},Y_{n+1}, \wh{\theta}_{n} \bigr)$. Since
$\left\| Z_{n+1} \right\| \leq \left\| \Phi_{n+1} \right\|$, it implies that
\begin{equation}
\label{TaylorG2}
G\bigl( \wh{\theta}_{n+1}\bigr) \leq G \bigl( \wh{\theta}_{n} \bigr) - \nabla G \bigl(  \wh{\theta}_{n} \bigr)^{T} S_n^{-1}Z_{n+1} +\frac{1}{8} \dE\bigl[ \left\| \Phi \right\|^{2} \bigr]
\left( \lambda_{\min} \left( S_{n} \right) \right)^{-2} \bigl\| \Phi_{n+1} \bigr\|^{2}.
\end{equation}
Let $\dF=(\cF_n)$ be the filtration given, for all $n \geq 1$, by 
$\cF_{n} =  \sigma \left( \left( \Phi_{1},Y_{1} \right),\ldots , \left( \Phi_{n},Y_{n} \right) \right)$. 
We clearly have $\dE\left[Z_{n+1}|\cF_{n} \right] = \nabla G \bigl( \wh{\theta}_{n} \bigr)$.
Consequently, we obtain from \eqref{TaylorG2} that
\begin{equation}
\label{TaylorG3}
\dE\bigl[ G\bigl( \wh{\theta}_{n+1}\bigr) | \cF_{n} \bigr]  \leq  G \bigl( \wh{\theta}_{n} \bigr) 
-\nabla G \bigl(  \wh{\theta}_{n} \bigr)^{T} S_n^{-1}   \nabla G \bigl(  \wh{\theta}_{n} \bigr) 
+ \frac{1}{8} \bigl(\dE\bigl[ \left\| \Phi \right\|^{2} \bigr]\bigr)^2\!\bigl( \lambda_{\min} \left( S_{n} \right)  \bigr)^{-2}
\hspace{0.5cm}\text{a.s.}
\end{equation}
Our goal is now to apply the Robbins-Siegmund theorem (\cite{Duf97}, page 18) to the three positive sequences $(V_n)$, $(A_n)$ and $(B_n)$ given by
$V_n= G \bigl( \wh{\theta}_{n} \bigr)$,
$$
A_n=\frac{1}{8} \bigl(\dE\bigl[ \left\| \Phi \right\|^{2} \bigr]\bigr)^2\bigl( \lambda_{\min} \left( S_{n} \right) \bigr)^{-2}
\hspace{1cm}\text{and}\hspace{1cm}
B_n=\bigl\| S_{n}^{-1/2} \nabla G \bigl(  \wh{\theta}_{n} \bigr) \bigr\|^{2} .
$$ 
It clearly follows from \eqref{TaylorG3} that
$$
\dE\bigl[ V_{n+1} | \cF_{n} \bigr]  \leq  V_n +A_n -B_n 
\hspace{1cm}\text{a.s.}
$$ 
Moreover, we already saw from \eqref{SumSn} that
\begin{equation}
\label{SumAn}
\sum_{n=1}^\infty A_n < \infty \hspace{1cm}
\text{a.s}
\end{equation}
Consequently, we can deduce from the Robbins-Siegmund theorem that $(V_n)$ convergences almost surely to a finite random variable and
\begin{equation}
\label{SumBn}
\sum_{n=1}^\infty B_n < \infty \hspace{1cm}
\text{a.s}
\end{equation}
Furthermore, since $B_n \geq \left(\lambda_{\max}\left( S_{n} \right) \right)^{-1} \bigl\| \nabla G \left(\wh{ \theta}_{n} \right) \bigr\|^{2}$, we get from 
\eqref{SumBn} that
\begin{equation}
\label{SumLmax1}
\sum_{n=1}^\infty \left( \lambda_{\max} \left( S_{n} \right) \right)^{-1} \bigl\| \nabla G \left( \wh{\theta}_{n} \right) \bigr\|^{2}  < \infty 
\hspace{1cm}\text{a.s}
\end{equation}
In addition, we obtain from \eqref{defalphan} together with \eqref{defSn} that
$$
\lambda_{\max} \left( S_{n} \right) \leq 1 + \frac{1}{4} \lambda_{\max} \Bigl( \sum_{k=1}^{n}\Phi_{k}\Phi_{k}^{T} \Bigr)
$$
since, for all $n \geq 1$, $\alpha_n \leq 1/4$. Therefore, \eqref{CvgSigman1} ensures that for $n$ large enough
$$
\lambda_{max}(S_n) \leq \Lambda n  \hspace{1cm}
\text{a.s}
$$
where $\Lambda$ is for the maximum eigenvalue of the positive definite deterministic matrix
$\dE\bigl[ \Phi \Phi^{T} \bigr] $. It implies that
\begin{equation}
\label{SumLmax2}
\sum_{n =1}^\infty \left( \lambda_{\max} \left( S_{n} \right) \right)^{-1} = + \infty  \hspace{1cm}\text{a.s}
\end{equation}
Hence, it follows from the conjunction of \eqref{SumLmax1} and \eqref{SumLmax2} that
$\nabla G \bigl( \wh{\theta}_{n} \bigr)$ converges to $0$ almost surely. It means that $\wh{\theta}_{n}$ converges almost surely to the unique zero $\theta$ of the gradient,
which is exactly what we wanted to prove. It remains to prove the almost sure convergence \eqref{T-ascvg2}. We infer from \eqref{defSn}
that
\begin{equation}
\label{DecOSn}
\overline{S}_{n} = 
\frac{1}{n}\sum_{k=1}^{n}\bigl( \alpha_{k} -\alpha \bigl( \Phi_{k},\wh{\theta}_{k-1} \bigr) \bigr)\Phi_{k}\Phi_{k}^{T} + 
\frac{1}{n}\sum_{k=1}^{n}\alpha \bigl( \Phi_{k} , \wh{\theta}_{k-1} \bigr) \Phi_{k}\Phi_{k}^{T} +
\frac{1}{n}I_{d+1}. 
\end{equation}
We now give the convergence of the two terms on the right-hand side of \eqref{DecOSn}. For the first one, one can observe that
$\alpha_{n} = \alpha \bigl( \Phi_{n} , \wh{\theta}_{n-1} \bigr)$ as soon as $\alpha \bigl( \Phi_{n} , \wh{\theta}_{n-1} \bigr) \geq c_{\alpha}n^{-\beta}$.
Consequently,
$$
\frac{1}{n}\sum_{k=1}^{n}\bigl( \alpha_{k} -\alpha \bigl( \Phi_{k},\wh{\theta}_{k-1} \bigr) \bigr)\Phi_{k}\Phi_{k}^{T} 
= \frac{1}{n}\sum_{k=1}^{n}\bigl( \alpha_{k} -\alpha \bigl( \Phi_{k},\wh{\theta}_{k-1} \bigr) \bigr)\Phi_{k}\Phi_{k}^{T}\,
\rI_{\bigl\{\alpha ( \Phi_{k} , \wh{\theta}_{k-1} ) \leq c_{\alpha}k^{-\beta}\bigr\}}.
$$
It implies that
$$
 \Bigl\|  \frac{1}{n} \sum_{k=1}^{n}\bigl( \alpha_{k} -\alpha \bigl( \Phi_{k},\wh{\theta}_{k-1} \bigr) \bigr)\Phi_{k}\Phi_{k}^{T}  \Bigr\| \leq
 \frac{c_\alpha}{n}\sum_{k=1}^{n} \frac{1}{k^{\beta}} \bigl\| \Phi_{k} \bigr\|^2.
$$
However, one can easily check from Lemma \ref{LemAbel}  that 
$$
\lim_{n\to \infty} \frac{1}{n}\sum_{k=1}^{n} \frac{1}{k^{\beta}}\bigl\| \Phi_{k} \bigr\|^2 = 0 
 \hspace{1cm}\text{a.s}
$$
It means that the first term on the right-hand side of \eqref{DecOSn} goes to $0$ almost surely.
We now study the convergence of the second term on the right-hand side of \eqref{DecOSn} which can be rewritten as
$$
\frac{1}{n}\sum_{k=1}^{n}\alpha \bigl( \Phi_{k} , \wh{\theta}_{k-1} \bigr) \Phi_{k}\Phi_{k}^{T}=
\frac{1}{n}\sum_{k=1}^{n}\alpha \bigl( \Phi_{k} , \theta \bigr) \Phi_{k}\Phi_{k}^{T}.
+\frac{1}{n}\sum_{k=1}^{n}\bigl( \alpha \bigl( \Phi_{k} , \wh{\theta}_{k-1} \bigr) - \alpha \bigl( \Phi_{k} , \theta \bigr) \bigr)\Phi_{k}\Phi_{k}^{T}
$$
On the one hand, thanks to the standard strong law of large numbers, we clearly have
\begin{equation}
\label{ascvgHG}
\lim_{n\to \infty} \frac{1}{n}\sum_{k=1}^{n} \alpha \bigl( \Phi_{k} , \theta \bigr) \Phi_{k}\Phi_{k}^{T}
= \dE\bigl[  \alpha \bigl( \Phi , \theta \bigr) \Phi\Phi^{T}\bigr]= \nabla^2G(\theta)
 \hspace{1cm}\text{a.s}
\end{equation}
On the other hand, denote by $R_n$ the remainder 
$$
R_n=\sum_{k=1}^{n}\bigl( \alpha \bigl( \Phi_{k} , \wh{\theta}_{k-1} \bigr) - \alpha \bigl( \Phi_{k} , \theta \bigr) \bigr)\Phi_{k}\Phi_{k}^{T}.
$$
We can split $R_n$ into two terms $R_n=P_n+Q_n$ where, for some positive constant $M$,
\begin{eqnarray*}
P_n & = & \sum_{k=1}^{n}\bigl( \alpha \bigl( \Phi_{k} , \wh{\theta}_{k-1} \bigr) - \alpha \bigl( \Phi_{k} , \theta \bigr) \bigr)\Phi_{k}\Phi_{k}^{T} \, \rI_{\bigl\{ \left\| \Phi_{k} \right\| \leq M \bigr\}} \\
Q_n & = & \sum_{k=1}^{n}\bigl( \alpha \bigl( \Phi_{k} , \wh{\theta}_{k-1} \bigr) - \alpha \bigl( \Phi_{k} , \theta \bigr) \bigr)\Phi_{k}\Phi_{k}^{T} \, \rI_{\bigl\{ \left\| \Phi_{k} \right\| > M \bigr\}}.
\end{eqnarray*}
It follows from the Lipschitz property of the function $\alpha$ given is Lemma \ref{Lemalpha} that
$$
\Bigl\| \frac{1}{n}P_{n}\Bigr\| \leq \frac{M}{12\sqrt{3}}\frac{1}{n}\sum_{k=1}^{n} \bigl\| \wh{\theta}_{k-1} - \theta \bigr\| \left\| \Phi_{k}\right\|^{2}.
$$
Hence, we deduce from \eqref{T-ascvg1} together with \eqref{CvgSigman1} that
\begin{equation}
\label{AscvgPn}
\lim_{n\to \infty} \frac{1}{n}P_n=0  \hspace{1cm}\text{a.s}
\end{equation}
Furthermore, we also have
$$
\Bigl\| \frac{1}{n}Q_{n}\Bigr\| \leq \frac{1}{2n}\sum_{k=1}^{n}  \left\| \Phi_{k}\right\|^{2} \, \rI_{\bigl\{ \left\| \Phi_{k} \right\| > M \bigr\}}.
$$
We deduce once again from the strong law of large numbers that
$$
\lim_{n\to \infty} \frac{1}{n} \sum_{k=1}^{n} \left\| \Phi_{k}\right\|^{2} \, \rI_{\bigl\{ \left\| \Phi_{k} \right\| > M \bigr\}}= \dE\bigl[ \left\| \Phi \right\|^{2} \, \rI_{\bigl\{ \left\| \Phi \right\| > M \bigr\}}\bigr]
\hspace{1cm}\text{a.s}
$$
which implies via \eqref{AscvgPn} that for any positive constant $M$,
\begin{equation}
\label{limsupRn}
\limsup_{n\to \infty} \Bigl\| \frac{1}{n}R_{n}\Bigr\| \leq \frac{1}{2} \dE\bigl[ \left\| \Phi \right\|^{2} \, \rI_{\bigl\{ \left\| \Phi \right\| > M \bigr\}}\bigr]
\hspace{1cm}\text{a.s}
\end{equation}
Nonetheless, we obtain from the Lebesgue dominated convergence theorem that
$$
\lim_{M \to \infty} \dE\bigl[ \left\| \Phi \right\|^{2} \, \rI_{\bigl\{ \left\| \Phi \right\| > M \bigr\}}\bigr]=0.
$$
Consequently, we find from \eqref{limsupRn} 
$$
\lim_{n\to \infty} \frac{1}{n}R_n=0  \hspace{1cm}\text{a.s}
$$
Finally, \eqref{T-ascvg2} follows from \eqref{DecOSn} and \eqref{ascvgHG}, which achieves the proof of Theorem \ref{T-ASCVG}.
\demend

\vspace{-2ex}

\subsection{Proof of Theorem \ref{T-RATE}.}

It follows from equation  \eqref{Defthetan} that for all $n \geq 1$,
\begin{equation*}
\wh{\theta}_{n+1}-\theta =   \wh{\theta}_{n} - \theta - \frac{1}{n} \Bigl( \overline{S}_{n}^{-1} -S^{-1} \Bigr)Z_{n+1}
- \frac{1}{n} S^{-1} Z_{n+1}
\end{equation*}
where $Z_{n+1}=\nabla_h g \bigl( \Phi_{n+1},Y_{n+1}, \wh{\theta}_{n} \bigr)$ and $S=\nabla^{2}G(\theta)$.
Consequently,
\begin{equation}
\label{DecR1}
\wh{\theta}_{n+1}-\theta=  \wh{\theta}_{n} - \theta - \frac{1}{n} \Bigl( \overline{S}_{n}^{-1} -S^{-1} \Bigr)Z_{n+1}
- \frac{1}{n} S^{-1} \Bigl( \nabla G \bigl(  \wh{\theta}_{n} \bigr) + \varepsilon_{n+1}\Bigr)
\end{equation}
where 
$\varepsilon_{n+1} = Z_{n+1} - \nabla G \bigl(  \wh{\theta}_{n} \bigr)$. We already saw that $\dE\left[Z_{n+1}|\cF_{n} \right] = \nabla G \bigl( \wh{\theta}_{n} \bigr)$
which clearly implies that $( \varepsilon_{n})$ is a martingale difference sequence, $\dE\left[\varepsilon_{n+1}|\cF_{n} \right] =0$.
Denote by $\delta_n$ the remainder of the Taylor's expansion of the gradient
$$
\delta_{n} = \nabla G \bigl(  \wh{\theta}_{n} \bigr)  - \nabla^{2}G(\theta) \bigl(  \wh{\theta}_{n} - \theta \bigr)=
\nabla G \bigl(  \wh{\theta}_{n} \bigr) - S \bigl(  \wh{\theta}_{n} - \theta \bigr).
$$
We deduce from \eqref{DecR1} that for all $n \geq 1$,
\begin{equation}
\label{DecR2}
\wh{\theta}_{n+1}-\theta=  \Bigl( 1- \frac{1}{n} \Bigr)\bigl(\wh{\theta}_{n} - \theta\bigr) - \frac{1}{n} \Bigl( \overline{S}_{n}^{-1} -S^{-1} \Bigr)Z_{n+1}
- \frac{1}{n} S^{-1} \Bigl( \delta_n + \varepsilon_{n+1}\Bigr),
\end{equation}
which leads to
\begin{eqnarray}
\wh{\theta}_{n+1}-\theta & = & - \frac{1}{n} \sum_{k=1}^n \Bigl( \overline{S}_{k}^{-1} -S^{-1} \Bigr)Z_{k+1}
- \frac{1}{n}S^{-1} \sum_{k=1}^n \Bigl( \delta_k + \varepsilon_{k+1}\Bigr), \nonumber \\
& = & - \frac{1}{n} \sum_{k=1}^n \Bigl( \overline{S}_{k}^{-1} -S^{-1} \Bigr)\Bigl( \varepsilon_{k+1} + \nabla G \bigl(  \wh{\theta}_{k} \bigr)\Bigr)
- \frac{1}{n}S^{-1} \sum_{k=1}^n \Bigl( \delta_k + \varepsilon_{k+1}\Bigr), \nonumber \\
& = & - \frac{1}{n} M_{n+1} -\Delta_{n}
\label{DecR3}
\end{eqnarray}
where
$$
M_{n+1}=\sum_{k=1}^n \overline{S}_{k}^{-1} \varepsilon_{k+1}
$$
and
\begin{eqnarray}
\Delta_{n} & = &  \frac{1}{n} \sum_{k=1}^n \Bigl( \overline{S}_{k}^{-1} -S^{-1} \Bigr)\nabla G \bigl(  \wh{\theta}_{k} \bigr)
+ \frac{1}{n}S^{-1} \sum_{k=1}^n  \delta_k, \nonumber \\
& = &  \frac{1}{n} \sum_{k=1}^n \Bigl( \overline{S}_{k}^{-1} -S^{-1} \Bigr)S\Bigl(  \wh{\theta}_{k} - \theta \Bigr)
+ \frac{1}{n} \sum_{k=1}^n \overline{S}_{k}^{-1} \delta_k.
\label{DecDeltan}
\end{eqnarray}
We claim that the remainder $\delta_n$ is negligeable. As a matter of fact,
\begin{eqnarray*}
\| \delta_{n} \|  &=& \Bigl\| \int_{0}^{1} \nabla^{2}G \bigl( \theta + t\bigl( \wh{\theta}_{n} - \theta \bigr) \bigr) \bigl( \wh{\theta}_{n} - \theta \bigr) dt - \nabla^{2} G( \theta) \bigl( \wh{\theta}_{n} - \theta 
\bigr) \Bigr\|,   \\
& \leq & \int_{0}^{1} \Bigl\| \nabla^{2}G \bigl( \theta + t\bigl( \wh{\theta}_{n} - \theta \bigr) \bigr)   - \nabla^{2} G( \theta) \Bigr\| dt \bigl\|  \wh{\theta}_{n} - \theta \bigr\|.
\end{eqnarray*}
However, the functional $G$ is twice continuously differentiable and $\wh{\theta}_{n}$ converges almost surely to  $\theta$, which ensures that
\begin{equation}
\label{DecR4}
\| \delta_{n} \| = o \bigl( \bigl \| \wh{\theta}_{n} - \theta \bigr\| \bigr) \hspace{1cm}
\text{a.s}
\end{equation}
Then, we obtain from \eqref{T-ascvg2},  \eqref{DecDeltan} and \eqref{DecR4} that it exists a constant $0<c<1/2$ and a finite positive random variable $D$ such that for all $n \geq 1$,
\begin{equation}
\label{DecR5}
\| \Delta_{n} \| \leq cL_n + \frac{1}{n}D \hspace{1cm}
\text{a.s}
\end{equation}
where
$$
L_n = \frac{1}{n} \sum_{k=1}^n \|  \wh{\theta}_{k}-\theta \|.
$$
Therefore, we deduce from \eqref{DecR3} and \eqref{DecR5} that for all $n \geq 1$,
\begin{eqnarray*}
L_{n+1}  &=& \Bigl(1- \frac{1}{n+1} \Bigl)L_{n} + \frac{1}{n+1} \|  \wh{\theta}_{n+1}-\theta \|, \\
&\leq & \Bigl(1- \frac{1}{n+1} \Bigl)L_{n} + \frac{1}{n+1} \Bigl( \frac{1}{n} \| M_{n+1} \| +  \| \Delta_{n}\|\Bigr) \hspace{1cm} \text{a.s} \\
&\leq & \Bigl(1- \frac{d}{n+1} \Bigl)L_{n} + \frac{1}{n(n+1)} \Bigl( \| M_{n+1} \| +  D\Bigr) \hspace{1cm} \text{a.s}
\end{eqnarray*}
where $d=1-c$. It clearly implies by induction that for all $n \geq 1$,
\begin{equation}
\label{DecR6}
L_n \leq \prod_{k=2}^n \Bigl(1 - \frac{d}{k}  \Bigr) L_1 
+ \sum_{k=2}^n \prod_{i=k+1}^n  \Bigl(1 - \frac{d}{i} \Bigr) \frac{1}{k(k+1)} \Bigl( \| M_{k+1} \| +  D\Bigr) \hspace{1cm} \text{a.s}
\end{equation}
Hereafter, we shall proceed to the evaluation of the right-hand side term in \eqref{DecR6}. The sequence $(M_n)$ 
is a locally square-integrable multi-dimensional martingale with predictable quadratic variation given by
\begin{equation}
\label{CVM}
\left\langle M \right\rangle_{n} = \sum_{k=2}^{n} \overline{S}_{k-1}^{-1} \dE \bigl[ \varepsilon_{k}\varepsilon_{k}^{T} |\cF_{k-1} \bigr] \overline{S}_{k-1}^{-1}.
\end{equation}
However, for all $n \geq 1$,
\begin{eqnarray*}
\dE\bigl[ \varepsilon_{n+1}\varepsilon_{n+1}^T |\cF_{n} \bigr] & = & \dE\Bigl[ Z_{n+1}Z_{n+1}^{T} |\cF_{n} \Bigr] - \nabla G \bigl(  \wh{\theta}_{n} \bigr) \nabla G \bigl(  \wh{\theta}_{n} \bigr)^T \\
& = & \dE\Bigl[ \Bigl( \pi \bigl( \widehat{\theta}_{n}^T\Phi_{n+1} \bigr) - Y_{n+1}\Bigr)^{2}\Phi_{n+1}\Phi_{n+1}^{T} |\cF_{n} \Bigr] - 
 \nabla G \bigl(  \wh{\theta}_{n} \bigr) \nabla G \bigl(  \wh{\theta}_{n} \bigr)^T \\
& = & \dE\Bigl[ \Bigl(  \pi \bigl( \widehat{\theta}_{n}^T\Phi_{n+1} \bigr) -  \pi \bigl( \theta^T\Phi_{n+1} \bigr) + \pi \bigl( \theta^T\Phi_{n+1} \bigr)
-Y_{n+1}\Bigr)^{2}\Phi_{n+1}\Phi_{n+1}^{T} |\cF_{n} \Bigr] \\
& - &  \nabla G \bigl(  \wh{\theta}_{n} \bigr) \nabla G \bigl(  \wh{\theta}_{n} \bigr)^T.
\end{eqnarray*}
Since $\dE\bigl[ Y_{n+1} |\Phi_{n+1} \bigr] = \pi \left( \theta^T\Phi_{n+1} \right)$, we obtain that for all $n \geq 1$,
\begin{eqnarray*}
\dE\bigl[ \varepsilon_{n+1}\varepsilon_{n+1}^T |\cF_{n} \bigr] & = & \dE\Bigl[ \Bigl(  \pi \bigl( \widehat{\theta}_{n}^T\Phi_{n+1} \bigr) -  \pi \bigl( \theta^T\Phi_{n+1} \bigr)\Bigr)^{2}\Phi_{n+1}\Phi_{n+1}^{T} |\cF_{n} \Bigr] \\
& + &  \dE\Bigl[ \Bigl(  \pi \bigl( \theta^T\Phi_{n+1} \bigr) -  Y_{n+1} )\Bigr)^{2}\Phi_{n+1}\Phi_{n+1}^{T} |\cF_{n} \Bigr] - \nabla G \bigl(  \wh{\theta}_{n} \bigr) \nabla G \bigl(  \wh{\theta}_{n} \bigr)^T \\
& = & \dE\Bigl[ \Bigl(  \pi \bigl( \widehat{\theta}_{n}^T\Phi_{n+1} \bigr) -  \pi \bigl( \theta^T\Phi_{n+1} \bigr)\Bigr)^{2}\Phi_{n+1}\Phi_{n+1}^{T} |\cF_{n} \Bigr]  + \nabla^{2}G ( \theta) \\
& - & \nabla G \bigl(  \wh{\theta}_{n} \bigr) \nabla G \bigl(  \wh{\theta}_{n} \bigr)^T.
\end{eqnarray*}
By continuity together with \eqref{T-ascvg1}, we have the almost sure convergences
$$
\lim_{n \rightarrow \infty} \nabla G \bigl(  \wh{\theta}_{n} \bigr) \nabla G \bigl(  \wh{\theta}_{n} \bigr)^T=0 \hspace{1cm} \text{a.s}
$$
and
$$
\lim_{n \rightarrow \infty} \dE\Bigl[ \Bigl(  \pi \bigl( \widehat{\theta}_{n}^T\Phi_{n+1} \bigr) -  \pi \bigl( \theta^T\Phi_{n+1} \bigr)\Bigr)^{2}\Phi_{n+1}\Phi_{n+1}^{T} |\cF_{n} \Bigr] = 0 \hspace{1cm} \text{a.s}
$$
Therefore, we obtain from \eqref{T-ascvg2} and \eqref{CVM} that
\begin{equation}
\label{LimCVM}
\lim_{n \to \infty} \frac{1}{n}\langle M \rangle_{n} = \bigl(\nabla^{2}G \left( \theta \right)\bigr)^{-1} \hspace{1cm} \text{a.s}
\end{equation}
Hence, it follows from the strong law of large numbers for multi-dimensional martingales given e.g. by
Theorem 4.13.16 in \cite{Duf97} that for any $\gamma >0$,
\begin{equation}
\label{DecR7}
\bigl\| M_n \bigr\|^2= o\bigl( n \bigl( \log n \bigr)^{1+\gamma}\bigr) \hspace{1cm} \text{a.s}
\end{equation}
Moreover, if the random vector $\Phi$ has a finite moment of order $>2$, we also have the more precise
almost sure rate of convergence
\begin{equation}
\label{DecR7Sharp}
\bigl\| M_n \bigr\|^2= O\bigl( n \log n\bigr) \hspace{1cm} \text{a.s}
\end{equation}
We will prove \eqref{T-rate2} inasmuch as the proof for \eqref{T-rate1} follows essentially the same lines.
We deduce from \eqref{DecR7Sharp} that it exists a finite positive random variable $C$ such that 
for all $n \geq 1$
\begin{equation}
\label{DecR8}
\bigl\| M_{n+1} \bigr\| \leq C \sqrt{n \log n} \hspace{1cm} \text{a.s}
\end{equation}
We are now in position to find an upper-bound for inequality \eqref{DecR6}.
Via the elementary $1-x \leq \exp(-x)$, we clearly have
$$
\prod_{k=2}^n \Bigl(1 - \frac{d}{k}  \Bigr) \leq \Bigl(\frac{2}{n+1}\Bigr)^d
 \hspace{1cm}\text{and} \hspace{1cm}
\prod_{i=k+1}^n \Bigl(1 - \frac{d}{i}  \Bigr) \leq \Bigl(\frac{k+1}{n+1}\Bigr)^d.
$$
Consequently, we obtain from \eqref{DecR6} and \eqref{DecR8} that for all $n \geq 1$,
$$
L_n  \leq   \Bigl(\frac{2}{n+1}\Bigr)^d L_1+ 
\sum_{k=2}^n \Bigl(\frac{k+1}{n+1}\Bigr)^d \frac{1}{k(k+1)} \Bigl( C  \sqrt{k \log k}+  D\Bigr) \hspace{1cm} \text{a.s}
$$
leading to
\begin{equation}
\label{DecR9}
L_n \leq   \Bigl(\frac{2}{n}\Bigr)^d L_1+ 
\frac{A\bigl( \log n \bigr)^{1/2}}{n^d}
\sum_{k=2}^n \frac{1}{k^{a}} \hspace{1cm} \text{a.s}
\end{equation}
where $A=\max(C,D)$ and $a=3/2 - d=1/2+c$. Hereafter, we recall that the positive constant $c$ has been chosen 
such that $c<1/2$ which means that $0<a<1$. Hence, we find from \eqref{DecR9} that for all $n \geq 1$,
\begin{eqnarray*}
L_n  & \leq &  \Bigl(\frac{2}{n}\Bigr)^d L_1+ 
\frac{A\bigl( \log n \bigr)^{1/2}}{n^{d+a-1}} \hspace{1cm} \text{a.s}\\
& \leq &  \Bigl(\frac{2}{n}\Bigr)^d L_1+ A\Bigl(
\frac{ \log n}{n}\Bigr)^{1/2} \hspace{1cm} \text{a.s}
\end{eqnarray*}
Since $d>1/2$, it immediately implies that
\begin{equation}
\label{DecR10}
L_n ^2 = O  \Bigl(\frac{ \log n}{n}\Bigr) \hspace{1cm} \text{a.s}
\end{equation}
Then, it follows from the conjunction of \eqref{DecR5} and \eqref{DecR10} that
\begin{equation*}
\| \Delta_{n} \|^2  = O \Bigl(\frac{\log n}{n} \Bigr) \hspace{1cm} \text{a.s}
\end{equation*}
It ensures, via \eqref{DecR3} and \eqref{DecR7Sharp}, that
\begin{equation*}
\| \wh{\theta}_{n}-\theta \|^2  = O \Bigl(\frac{\log n}{n} \Bigr) \hspace{1cm} \text{a.s}
\end{equation*}
which is exactly what we wanted to prove.
\demend

\vspace{-2ex}

\subsection{Proof of Theorem \ref{T-RATEH}.}

First of all, it follows from \eqref{defSn} that $\overline{S}_{n}$ can be splitted  into two terms
\begin{eqnarray}
\overline{S}_{n} &  = & \frac{1}{n}\sum_{k=1}^{n} \alpha_{k}\Phi_{k}\Phi_{k}^T + \frac{1}{n}I_{d+1} \nonumber\\
& = & \frac{1}{n} T_n +  \frac{1}{n}\sum_{k=1}^{n} \dE\bigl[\alpha_{k}\Phi_{k}\Phi_{k}^T | \cF_{k-1} \bigr] + \frac{1}{n}I_{d+1} 
\label{DecRH1}
\end{eqnarray}
where
$$
T_n = \sum_{k=1}^{n} \alpha_{k}\Phi_{k}\Phi_{k}^T - \dE\bigl[\alpha_{k}\Phi_{k}\Phi_{k}^T | \cF_{k-1} \bigr].
$$
The sequence $(T_n)$ is a locally square-integrable multi-dimensional martingale. Since the random vector $\Phi$ has a finite moment
of order $4$ and for all $n \geq 1$, $\alpha_n \leq 1/4$, we obtain from the strong law of large numbers for multi-dimensional martingales given e.g. by
Theorem 4.13.16 in \cite{Duf97} that for any $\gamma >0$,
\begin{equation}
\label{DecRH2}
\bigl\| T_n \bigr\|^2= o\bigl( n \bigl( \log n \bigr)^{1+\gamma}\bigr) \hspace{1cm} \text{a.s}
\end{equation}
Let us now give the rate of convergence of the second term on the right-hand side of \eqref{DecRH1}. 
On the one hand, we have the decomposition
$$
\sum_{k=1}^{n} \alpha_{k}\Phi_{k}\Phi_{k}^T= \sum_{k=1}^{n}\bigl( \alpha_{k} -\alpha \bigl( \Phi_{k},\wh{\theta}_{k-1} \bigr) \bigr)\Phi_{k}\Phi_{k}^{T}+
\sum_{k=1}^{n}\alpha \bigl( \Phi_{k},\wh{\theta}_{k-1} \bigr) \Phi_{k}\Phi_{k}^{T}.
$$
We already saw in the proof of Theorem \ref{T-ASCVG} that
\begin{equation}
\label{DecRH3}
\sum_{k=1}^{n}\bigl( \alpha_{k} -\alpha \bigl( \Phi_{k},\wh{\theta}_{k-1} \bigr) \bigr)\Phi_{k}\Phi_{k}^{T}\leq 
c_\alpha\sum_{k=1}^{n} \frac{1}{k^{\beta}} \Phi_{k}\Phi_{k}^{T}
\end{equation}
Hence, by taking the conditionnal expectation on both sides of \eqref{DecRH3}, we obtain that
\begin{eqnarray}
\Bigl \| \sum_{k=1}^{n} \dE\Bigl[\bigl( \alpha_{k} -\alpha \bigl( \Phi_{k},\wh{\theta}_{k-1} \bigr) \bigr)\Phi_{k}\Phi_{k}^{T} | \cF_{k-1} \Bigr] \Bigr \|
& \leq & c_\alpha\sum_{k=1}^{n} \frac{1}{k^{\beta}} \dE \bigl[ \|\Phi_{k}\|^{2} |  \cF_{k-1}\bigr] \nonumber   \\
& \leq &  \frac{c_\alpha \dE \bigl[ \|\Phi\|^{2}] }{1-\beta} n^{1-\beta}.
\label{DecRH4}
\end{eqnarray}
On the other hand,
$$
\sum_{k=1}^{n}\dE\Bigl[ \alpha \bigl( \Phi_{k},\wh{\theta}_{k-1} \bigr) \Phi_{k}\Phi_{k}^{T}  | \cF_{k-1} \Bigr]
= \sum_{k=1}^n  \Bigl( \nabla^{2}G \bigl( \wh{\theta}_{k-1} \bigr) -  \nabla^{2}G \bigl( \theta \bigr) \Bigr) + n \nabla^{2}G \bigl( \theta \bigr).
$$
Consequently, we immediately deduce from inequality \eqref{Lipalpha} that
\begin{eqnarray*}
\Bigl \| \sum_{k=1}^{n}\dE\Bigl[ \alpha \bigl( \Phi_{k},\wh{\theta}_{k-1} \bigr) \Phi_{k}\Phi_{k}^{T}  | \cF_{k-1} \Bigr] - n \nabla^{2}G \bigl( \theta \bigr) \Bigr\|
& \leq & \sum_{k=1}^n  \Bigl\| \nabla^{2}G \bigl( \wh{\theta}_{k-1} \bigr) -  \nabla^{2}G \bigl( \theta \bigr) \Bigr\|  \\
& \leq & \frac{1}{12 \sqrt{3} }\sum_{k=1}^{n} \bigl \| \wh{\theta}_{k-1} - \theta \bigr \| \dE \bigl[ \|\Phi_{k}\|^{3} |  \cF_{k-1}\bigr] 
\end{eqnarray*}
which implies that
\begin{equation}
\label{DecRH5}
\Bigl \| \sum_{k=1}^{n}\dE\Bigl[ \alpha \bigl( \Phi_{k},\wh{\theta}_{k-1} \bigr) \Phi_{k}\Phi_{k}^{T}  | \cF_{k-1} \Bigr] - n \nabla^{2}G \bigl( \theta \bigr) \Bigr\|
\leq \frac{\dE \bigl[ \|\Phi\|^{3}]}{12 \sqrt{3} } \sum_{k=1}^{n} \bigl \| \wh{\theta}_{k-1} - \theta \bigr \| 
\end{equation}
Finally, it follows from the conjunction of \eqref{DecRH2},  \eqref{DecRH4} and  \eqref{DecRH5} together with \eqref{DecR10}
that for all $0<\beta<1/2$,
\begin{equation*}
 \bigl\| \overline{S}_{n} - \nabla^{2}G \left( \theta \right) \bigr\|^{2} = O \left( \frac{1}{n^{2\beta}} \right) \hspace{1cm}\text{a.s}
\end{equation*}
which achieves the proof of \eqref{T-rate3}. Moreover, we obtain \eqref{T-rate4} from \eqref{T-rate3}
via the identity
$$
\overline{S}_{n}^{-1} - \left( \nabla^{2}G \left( \theta \right) \right)^{-1}=
\overline{S}_{n}^{-1}\left( \nabla^{2}G \left( \theta \right)  - \overline{S}_{n} \right) \left( \nabla^{2}G \left( \theta \right) \right)^{-1}.
$$

\section{Proofs of the asymptotic normality result}
\label{secproofsan}
We are now in the position to proceed to the proof of the asymptotic normality \eqref{T-an1}. We clearly have from
\eqref{DecR3} that

\begin{equation}
\label{Pran1}
\sqrt{n} \bigl( \wh{\theta}_{n+1}-\theta \bigr) = - \frac{1}{\sqrt{n}} M_{n+1} - R_n
\end{equation}
where the remainder $R_n=\sqrt{n} \Delta_{n}$. First of all, we claim that 
\begin{equation}
\label{Pran2}
\lim_{n \rightarrow \infty} R_n = 0  \hspace{1cm} \text{a.s}
\end{equation}
As a matter of fact, it follows from \eqref{DecDeltan} that $R_n=P_n+Q_n$ where
\begin{eqnarray*}
P_n & = & \frac{1}{\sqrt{n}} \sum_{k=1}^n \Bigl( \overline{S}_{k}^{-1} -S^{-1} \Bigr)S\Bigl(  \wh{\theta}_{k} - \theta \Bigr), \\
Q_n & = & \frac{1}{\sqrt{n}} \sum_{k=1}^n \overline{S}_{k}^{-1} \delta_k.
\end{eqnarray*}
We have from \eqref{T-rate2} together with \eqref{T-rate4} that
\begin{equation*}
\| P_n \|=O \left(\frac{1}{\sqrt{n}} \sum_{k=1}^n \frac{1}{k^\beta} \frac{\sqrt{\log k }}{\sqrt{k}}\right)=O \left(\frac{\sqrt{\log n}}{n^\beta}\right)
\hspace{1cm} \text{a.s}
\end{equation*}
which implies that
\begin{equation}
\label{Pran3}
\lim_{n \rightarrow \infty} P_n = 0  \hspace{1cm} \text{a.s}
\end{equation}
Moreover, we obtain from inequality \eqref{Lipalpha} and \eqref{T-rate2} that
\begin{equation*}
\| Q_n \|=O \left(\frac{1}{\sqrt{n}} \sum_{k=1}^n \bigl\| \wh{\theta}_{k}-\theta \bigr\|^2\right) = O \left(\frac{1}{\sqrt{n}} \sum_{k=1}^n \frac{\log k }{k}
\right) = O \left(\frac{(\log n)^2}{\sqrt{n}}\right)
\hspace{1cm} \text{a.s}
\end{equation*}
which also implies that
\begin{equation}
\label{Pran4}
\lim_{n \rightarrow \infty} Q_n = 0  \hspace{1cm} \text{a.s}
\end{equation}
Consequently, \eqref{Pran3} and \eqref{Pran4} clearly lead to convergence \eqref{Pran2}. Hereafter, it only remains to study the asymptotic behavior
of the martingale term $M_n$. We already saw from \eqref{LimCVM} that its predictable quadratic variation $\langle M \rangle_{n}$ satisfies
\begin{equation*}
\lim_{n \to \infty} \frac{1}{n}\langle M \rangle_{n} = \bigl(\nabla^{2}G \left( \theta \right)\bigr)^{-1} \hspace{1cm} \text{a.s}
\end{equation*}
In addtion, as $\varepsilon_{n+1} = \nabla_h g \bigl( \Phi_{n+1},Y_{n+1}, \wh{\theta}_{n} \bigr) - \nabla G \bigl(  \wh{\theta}_{n} \bigr)$, we clearly have
the very simple upper-bound
$$
\bigl\| \varepsilon_{n+1} \bigr\| \leq \bigl\| \Phi_{n+1} \bigr\| + \dE\bigl[ \bigl\| \Phi \bigr\| \bigr] .
$$
Hence, since $\Phi$ has a finite moment of order $4$, 
\begin{equation}
\label{Pran5}
\sup_{n \geq 1} \dE \bigl[ \bigl\| \varepsilon_{n} \bigr\|^4 \bigr] < \infty
\end{equation}
Therefore, we immediately obtain from \eqref{Pran5} that $(M_n)$ satisfies Lindeberg's condition.
Finally, we deduce from the central limit theorem for martingales given by Corollary 2.1.10 in  \cite{Duf97}
that
$$
\frac{1}{\sqrt{n}} M_{n} \cvl \mathcal{N}\Bigl( 0 , \left(\nabla^{2} G \left( \theta \right) \right)^{-1} \Bigr)
$$
which, via \eqref{Pran1} and \eqref{Pran2}, completes the proof of Theorem \ref{T-AN}.
\demend

\vspace{-2ex}
\bibliographystyle{abbrv}
\def\cprime{$'$}

\end{document}